\documentclass{amsart}

\usepackage{amsmath,amssymb,epsfig,amsthm} % RL removed amsthm on 1/17/18 because it prevented compile
\usepackage[usenames,dvipsnames]{xcolor}
\usepackage{hyperref} %to display hyperlinks
\usepackage{verbatim} %enables comment environment
\usepackage{color}     %for RL's commenting
\usepackage{pdflscape}  %for the landscaped tables
\usepackage[curve]{xypic} %for ram portraits
\usepackage{fullpage} %this includes more text on each page: we can delete this later, but  it makes writing easier
\usepackage{colortbl} %to shade rows in the table
\usepackage{multirow} %for merging cells
\usepackage{caption}
\newcommand\Z{\mathbb Z}

\newcommand\N{\mathbb N}

\newcommand\sphere{\mathbb S^2}
\newcommand\rs{\widehat{\mathbb{C}}}

\newcommand\phibar{\bar{\phi}}

\newtheorem{theorem}{Theorem}[section]
\newtheorem{proposition}[theorem]{Proposition}
\newtheorem{lemma}[theorem]{Lemma}
\newtheorem{corollary}[theorem]{Corollary}
\theoremstyle{definition}
\newtheorem{definition}[theorem]{Definition}

\definecolor{reallylightgray}{gray}{0.90}

\newtheorem{remark}[theorem]{Remark}

\begin{document}

\title{Quadratic Thurston maps with few postcritical points}

\begin{author}[Kelsey]{Gregory Kelsey}
\email{gkelsey@bellarmine.edu}
\address{Bellarmine University, 2001 Newburg Rd, Louisville, KY 40205}
\end{author}

\begin{author}[Lodge]{Russell Lodge}
\email{Russell.Lodge@indstate.edu}
\address{Department of Mathematics and Computer Science, Indiana State University, Terre Haute, IN 47809}
\end{author}


\thanks{Both authors gratefully acknowledge the support of the AIM SQuaRE program 2013-2015. 
The second author was also supported by the Deutsche Forschungsgemeinschaft.
The authors also thank Kevin Pilgrim, Walter Parry, and the anonymous referee for their helpful comments on early drafts of this article. 
Sarah Koch and Bill Floyd provided valuable perspectives and assistance, and the authors thank them as well. }

\maketitle
%\date{\today}

\begin{abstract}
We use the theory of self-similar groups to enumerate all combinatorial classes of non-Euclidean quadratic Thurston maps with fewer than five postcritical points. 
The enumeration relies on our computation that the corresponding maps on moduli space can be realized by quadratic rational maps with fewer than four postcritical points.

\end{abstract}

\section{Introduction}
\label{sec:intro}

Holomorphic dynamics in one complex variable is largely concerned with the study of rational maps as dynamical systems, along with their parameter spaces.   Postcritically finite rational maps have attracted much attention due to their relative simplicity and their structural significance in parameter space.  Furthermore, W. Thurston has given a powerful characterization and rigidity theorem for postcritically finite rational maps up to combinatorial equivalence (see Theorem \ref{thm:Thurston'sThm}). With this tool at our disposal, we are free to consider rational maps in the more flexible topological category.

\begin{definition}[Thurston map]
Let $f:\sphere\to\sphere$ be an orientation-preserving branched cover of degree $d\geq 2$ whose set of critical points is denoted $C_f$.  Then the \emph{postcritical set} of $f$ is given by $P_f:=\bigcup_{i>0} f^{\circ i}(C_f)$. If $P_f$ is finite, $f$ is said to be a \emph{Thurston map}.  
\end{definition}

Thurston maps will be considered up to the following equivalence relation, which can roughly be thought of as ``conjugacy up to homotopy."

\begin{definition}[Combinatorial equivalence]
Two Thurston maps $f$ and $g$ are said to be \emph{combinatorially equivalent} if and only if there exist orientation-preserving homeomorphisms $h_0, h_1:(\sphere,P_f)\to(\sphere,P_g)$ so that
\[h_0\circ f= g\circ h_1\]
and $h_0$ is isotopic to $h_1$ through a one parameter family of homeomorphisms $h_t:(\sphere,P_f)\to(\sphere,P_g)$ where $0\leq t\leq 1$.
\end{definition}

% Stolen from Thurston Theory section

A Thurston map is called \emph{Euclidean} if its orbifold Euler characteristic is zero (cf. \cite[Definition 1.1]{CFPP}). A quadratic Thurston map with four postcritical points is Euclidean if and only if it has no periodic critical point. 

We will state Thurston's theorem in the special case when $f$ is a non-Euclidean Thurston map with $|P_f|=4$ (the general statement is found in \cite{DH93}). A simple closed curve in $\sphere$ is said to be \emph{essential} if it separates the postcritical set into two pairs of points.  Two essential curves are said to be \emph{homotopic} if they are homotopic in $\sphere\setminus P_f$. The collection of all homotopy classes of essential curves is denoted $\mathcal{C}_f$. 

Let $\tilde{\gamma}_1,...,\tilde{\gamma}_k$ be the collection of essential components of $f^{-1}(\gamma)$, and let $d_i$ denote the mapping degree of $\tilde{\gamma_i}\rightarrow\gamma$. Then $\gamma\in\mathcal{C}_f$ is said to be an \emph{obstructing curve for $f$} if $\gamma$ is homotopic to some essential component of $f^{-1}(\gamma)$ and $\sum_{i=1}^{k}\frac{1}{d_i}\geq 1$.

\begin{theorem}[Thurston's theorem for $|P_f|=4$]\label{thm:Thurston'sThm}
Let $f$ be a Thurston map with four postcritical points that is not Euclidean.  Then $f$ is equivalent to a rational map if and only if there is no obstructing curve.  If the rational map exists, it is unique up to M\"obius conjugacy.
\end{theorem}

Applications of the general statement of Thurston's theorem include the combinatorial classifications of large families of postcritically finite rational maps \cite{Poirier,Newton}.

% End Theft

It is natural to seek an enumeration of all combinatorial classes (see for instance Problem 5.2 posed by McMullen \cite{McM}). Few explicit listings exist, with the following exceptions.  First is the enumeration of all non-polynomial hyperbolic rational maps of degree 3 or less with four or fewer postcritical points \cite{KMPcensus}.  This is done by listing all possible dynamical portraits (i.e. combinatorial descriptions of critical and postcritical dynamics), and then finding coefficients of the corresponding rational functions by solving equations derived from the portraits. However, enumeration is substantially more difficult in the presence of combinatorial classes with no rational representative. The main results of this type come in the context of quadratic and cubic Thurston maps with four postcritical points \cite{K01,BN06,Lod}.  For example, Bartholdi and Nekrashevych enumerate all maps with the same dynamical portrait as $z^2+i$ (see Figure \ref{ramportfig}) using special features of self-similar groups associated to polynomials. It is known that a large family of polynomial dynamical portraits admit obstructed combinatorial classes \cite{Kel}. 

\begin{figure}[h]
$$\xymatrix{0\ar@{=>}[r] & i\ar[r] & -1+i \ar@/^/[r] & -i \ar@/^/[l] &\infty\ar@{=>}@(ur,dr)[]}$$
\caption{The dynamical portrait of $z^2+i$, which is realized by both obstructed and unobstructed Thurston maps}
\label{ramportfig}
\end{figure}

The main goal of the present work is to prove the following (see \S \ref{sec:ids}):

\begin{theorem}[Enumeration of combinatorial classes]
\label{thm:classification}
Let $f$ be a quadratic non-Euclidean Thurston map with four or fewer postcritical points. If $f$ has
\begin{itemize}
\item three or fewer postcritical points, then it is equivalent to a map listed in Table \ref{tbl:Quad3}.
\item four postcritical points and no obstructing curve, then it is equivalent to a map listed in Table \ref{tbl:Hurwitz1}.
\item four postcritical points and is not equivalent to a rational map, then it is equivalent to a map in Table \ref{tbl:Obstructed}.
\end{itemize}
\end{theorem}

\subsection{Reading the Tables}
%\rbl{if the reader jumps to the tables from here, they won't know the M\"obius notation. Define here? Or reference? The point pushing T notation is ``defined" in the caption for the obstructed maps. Describe here too?}
Table \ref{tbl:Quad3} (on page \pageref{tbl:Quad3}) deals with the case of three or fewer postcritical points. Such Thurston maps are all equivalent to rational maps, so in this case our enumeration amounts to listing all possible dynamical portraits (subject to Riemann-Hurwitz and local degree restrictions) and computing coefficients. Our table lists the dynamical portraits and fixed points of these maps and organizes them by the number of critical postcritical points because we will use this information in our enumeration of non-Euclidean Thurston maps with exactly four postcritical points.

Table \ref{tbl:Hurwitz1} (on page \pageref{tbl:Hurwitz1}) deals with non-Euclidean Thurston maps with exactly four postcritical points that do not admit obstructing curves. Each of these maps can be decomposed as a map $g$ with three or fewer postcritical points from Table \ref{tbl:Quad3} postcomposed by a M\"obius transformation 
  $M_{\bullet,0}, M_{\bullet,1},$ or $M_{\bullet,\infty}$ that respectively induce the following permutations: \[ (\bullet\; 0)(1\;\infty),\qquad (\bullet\; 1)(0\;\infty),\qquad (\bullet \;\infty)(0\;1)\]
for some  $\bullet\in\rs\setminus\{0,1,\infty\}$ that is specified in the table (see also Proposition \ref{prop:RationalQ4Comps}). We further show in Theorem \ref{Thm:mapOnModuliSpace} that the map on moduli space associated to the Thurston map $f$ is in fact $g$ (our use of the notation $g$ for the map on moduli space is consistent with \cite{K} in this sense). %Table \ref{tbl:Hurwitz1} then lists the dynamical portraits for our rational maps, as well as the fixed point of the map on moduli space (the $g$-map) they are associated with. The exception is $g(z)=z^2$, which admits no rational representatives. \rbl{last sentence needs modification}

Table \ref{tbl:Obstructed} (on page \pageref{tbl:Obstructed}) deals with the obstructed cases. Apart from the exceptional case of $g(z)=z^2$ which requires a little more care, we take a rational map presented in terms of the same data as before ($g$-map, M\"obius transformation, and fixed point of $g$) and list the twists that can be applied to the rational map to obtain distinct combinatorial classes of obstructed maps. These twists are defined in terms of loops $\alpha$, $\beta$ in the moduli space (see $\S$\ref{sec:wreath} for the choice of these loops and the point-pushing isomorphism that identifies them with twists).

%\gak{Should we include a discussion of how to read the nucleus and FGA tables?}

\subsection{Outline}

The remainder of this paper is organized as follows:

In $\S$\ref{sec:Moduli} we develop the Teichm\"uller theory to show that certain rational maps with small postcritical set can be composed with a M\"obius transformation to obtain their corresponding map on moduli space (Theorem \ref{Thm:mapOnModuliSpace}). We call such maps M\"obius reducible.
In particular, this theorem provides a  sufficient condition for the pullback map of a rational marked cover with exactly four marked points to cover a map (for general reasons one only expects a correspondence) on moduli space. Such results were previously limited to polynomials with periodic critical points and polynomials with exactly one critical point in the complex plane \cite[Theorems 5.17, 5.18]{K}.

In $\S$\ref{sec:RamPort} we specialize to the quadratic case. We enumerate the dynamical portraits of the maps under investigation and show that all rational quadratics with exactly four postcritical points are M\"obius reducible (Proposition \ref{prop:RationalQ4Comps}).

%In \S\ref{sec:StatPort} we make a brief aside to demonstrate how the theory of M\"obius reducible maps generalizes outside of the quadratic setting. In particular, we give a (Theorem \ref{thm:static}). .% or degree two maps \cite{Jung}.

In $\S$\ref{sec:ssgrps} we provide background on self-similar group theory and introduce the fundamental algebraic objects and actions that we will use to produce our enumeration.

In $\S$\ref{sec:wreath}, we use the $g$-map and M\"obius transformation decomposition to calculate the self-similar action for each of our rational maps in Tables \ref{tbl:Hurwitz1}. For each action we then compute the nucleus, which is a set that we will use for our calculations in the next section. An important result is that each of these actions is {\it sub-hyperbolic}, which means that each nucleus contains only finitely-many distinct actions (Theorem \ref{thm:subhyp}).

The mapping class bisets we use in our calculations are also the central object in Bartholdi and Dudko's work on the decidability of combinatorial equivalence \cite{BD16}.
These bisets are interesting objects in their own right. 
The question of which Thurston maps give rise to sub-hyperbolic mapping class bisets was asked in 
\cite{N09}, and our work provides the first answer for non-polynomial maps. 

In $\S$\ref{sec:TwistedNET} we prove the main theorem.

For non-Euclidean maps with four postcritical points, we rely on a crucial fact about the transitivity of the mapping class action. A \emph{twist} of $f$ is a map of the form $h \circ f$ where $h:(\sphere,P_f)\to(\sphere,P_f)$ is an orientation-preserving homeomorphism with $h|_{P_f}=id$; note that $h\circ f$ is again a Thurston map with the same dynamical portrait. A special feature of quadratic Thurston maps with four postcritical points is that every combinatorial class having the same dynamical portrait as $f$ can be represented by a twist of $f$ (Lemma \ref{Hurwitzlem}).

We present an algorithm to determine the combinatorial class of some twist $h_0$ of $f$.  Recall that the pure mapping class group of $(\sphere,P_f)$ is isomorphic to the free group on two generators.  Following \cite{BN06}, we extend the virtual endomorphism on the pure mapping class group and iterate. This defines a sequence of mapping classes $\{h_i\}_{i=0}^\infty$ where $h_i\circ f$ is combinatorially equivalent to $h_{i+1}\circ f$ for all $i\geq 0$.  
A set built from the nucleus of the self-similar action of the pure mapping class group on combinatorial classes contains the tails of all such sequences $\{h_i\}_{i=0}^\infty$.
%We prove that the mapping class biset is {\it sub-hyperbolic}, which means that the nucleus is either finite or has only finitely-many distinct actions of the mapping class group.
Using our nucleus computations from the previous section, we show that for any map under consideration, the tail of a twisting sequence consists of a finite number of cycles of the extended virtual endomorphism together with at most two $\mathbb{Z}$-parameter families of fixed points. With a few exceptions, each of these cycles corresponds to a distinct combinatorial class, and we identify rational representatives (if they exist) using biset invariants such as the pullback function on curves discussed below. 

In $\S$\ref{sec:PullbackOnCurves} we use the self-similar calculations for these maps to determine the pullback of essential curves. Following the methods of \cite{TW} and \cite{Lod}, we extend the virtual endomorphism on the pure mapping class group and iterate (note that this extension is different from the one used in the previous section for solving the twisting problem). This extension allows us to compute the backward dynamics of homotopy classes of essential curves under the rational maps we are studying. Differences in these dynamics provide an invariant that we use to distinguish combinatorial classes in our enumeration. Additionally, we prove that each rational map in this category eventually pulls every curve back onto a finite set (Theorem \ref{thm:FGA}).

\section{M\"obius reducibility and maps on moduli space}
\label{sec:Moduli}

In this section we prove results for Thurston maps with four postcritical points, but arbitrary degree. The fundamental idea is that a rational Thurston map $g$ with fewer than four postcritical points can be postcomposed by a M\"obius transformation to produce a marked branched cover (a mild generalization of Thurston map) with exactly four marked points (Proposition \ref{prop:postcompOfRational}). We then give a non-dynamical sufficient condition for a Thurston map to arise as a composition in this way (Proposition \ref{prop:cleanImpliesReducible}). An important consequence of this decomposition is that the pullback map on Teichm\"uller space for this marked branched cover has an associated map on moduli space that is $g$ in appropriate coordinates (see Theorem \ref{Thm:mapOnModuliSpace}). We first define the relevant terminology following \cite{BCT}.

\subsection{Marked branched covers}

A \emph{marked cover} is a pair $(f,Y)$ where $f:\sphere\to\sphere$ is an orientation-preserving branched cover with degree $d\geq 2$ and $Y\subset\sphere$ is a finite set containing the postcritical set $P_f$ and satisfying $f(Y)\subset Y$. The pair $(f,Y)$ is said to be \emph{rational} if $f$ is rational. Every Thurston map $f$ is considered to be a marked cover with marked set $P_f$ unless otherwise stated. We now define some invariants for marked covers.

\begin{definition}[Dynamical and static portraits]
Let $(f,Y)$ be a marked cover.

The \emph{dynamical portrait} of $(f,Y)$ is the labeled directed graph whose vertices are points in $Z_f:=Y\cup C_f$.  A directed edge connects vertex $v$ to vertex $w$ if and only if $f(v)=w$. The label of such an edge is given by the local degree of $f$ at $v$.

The {\it static portrait} of $f$ is a labeled directed graph with directed edges between two disjoint sets: directed edges originate in vertices $Y$ and terminate in a disjoint copy of vertices $P_f$. A directed edge connects vertex $v$ to vertex $w$ if and only if $f(v)=w$. The label of such an edge is given by the local degree of $f$ at $v$.
\end{definition}

Intuitively, the static portrait of a marked cover is the dynamical portrait with the domain and range no longer identified.

We say that marked covers $(f,Y_f)$ and $(g,Y_g)$ have \emph{equivalent portraits} (dynamical or static) if and only if there is a bijection $h:Z_f\rightarrow Z_g$ so that $h\circ f=g\circ h$ on $Z_f$ and $h$ preserves edge labels (i.e. the local degree of $f$ at $z\in Z_f$ is identical to that of $g$ at $h(z)$). If Thurston maps $f$ and $g$ are combinatorially equivalent, they evidently have equivalent portraits.

In Section \ref{subsec:mapsOnModuli} we will show that the iteration on Teichm\"uller space associated to certain marked covers takes a specific form. We define the relevant terminology here.
Suppose $Y=\{y_0,y_1,y_2,y_3\}\subset\sphere$. Then the \emph{moduli space} of $\sphere$ with marked set $Y$ is given by
\[\mathcal{M}_Y:=\{\iota:Y\hookrightarrow\rs\}/\sim\]
where two injections $\iota_1,\iota_2:Y\hookrightarrow\rs$ are equivalent if and only if there exists a M\"obius transformation $M:\rs\to\rs$ so that $M\circ \iota_1=\iota_2$. Then \emph{Teichm\"uller space} of $\sphere$ with marked set $Y$ is given by 
\[\mathcal{T}_Y:=\{\text{homeomorphisms }\phi:\sphere\to\rs\}/\sim\]
where two homeomorphisms $\phi_1,\phi_2:\sphere\to\rs$ are equivalent if and only if there exists a M\"obius transformation $M:\rs\to\rs$ so that $M\circ \phi_1|_Y=\phi_2|_Y$ and $M\circ \phi_1$ is isotopic to $\phi_2$ rel $Y$. There is an obvious analytic projection $\pi:\mathcal{T}_Y\to\mathcal{M}_Y$ defined by $\pi([\phi])=[\phi|_Y]$.

The \emph{pullback map on Teichm\"uller space} $\sigma_{f,Y}:\mathcal{T}_Y\to\mathcal{T}_Y$ associated to a marked cover $(f,Y)$ is defined here (it was used to generalize Thurston's theorem to the setting of marked branched covers \cite[Theorem 2.1]{BCT}). Let $\phi$ represent a point $\tau\in\mathcal{T}_Y$. Pull back the standard complex structure on $\rs$ under the map $f\circ\phi:\sphere\to\rs$. Uniformization of this new complex structure defines a map $\tilde{\phi}:\sphere\to\rs$ that represents some point $\tilde{\tau}\in\mathcal{T}_Y$. The pullback map is defined by $\sigma_{f,Y}(\tau)=\tilde{\tau}$ and is shown in \cite{BCT} to be well-defined and analytic. It is known that if $f$ is a non-Euclidean rational map, the map $\sigma_{f,Y}$ has a unique attracting fixed point in $\mathcal{T}_Y$ \cite[Theorem 2.2]{BCT}.

\subsection{Static reducibility}

Many of the objects in this section are considered up to automorphism of $\rs$. We will normalize three marked points to be $0,1,$ and $\infty$ for convenience, and lose no generality in doing so. Let $A=\{0,1,\infty,a\}\subset\rs$ be a set of four points.

\begin{proposition}[Marked covers from M\"obius composition]\label{prop:postcompOfRational}
Let $g:\rs\to\rs$ be a rational map with $P_g\subset\{0,1,\infty\}$, and suppose that $g(A)\subset A$. If $M$ is a M\"obius transformation with $M(A)=A$, then $(M\circ g,A)$ is a rational marked cover.
\end{proposition}

\begin{proof}
Each critical value of $g$ is contained in $P_g$ which is a subset of $A$. Thus each critical value of $M\circ g$ is in $A$. Since $A$ is forward invariant under $M\circ g$, it follows that the postcritical set of $M\circ g$ is in $A$.
\end{proof}

This section develops the properties of marked covers that arise in this way.

\begin{definition}[M\"obius reducible]\label{defn_MobReducible} We say that a rational marked cover $(f,A)$ with $A=\{0,1,\infty,a\}$ is \emph{M\"obius reducible} if $f=M\circ g$ for some M\"obius transformation $M$ and some rational marked cover $(g,B)$ with $B\subset\{0,1,\infty\}$.
\end{definition}

In Proposition \ref{prop:cleanImpliesReducible} we will give a sufficient condition for a rational marked cover to be M\"obius reducible. As such it is a partial converse to Proposition \ref{prop:postcompOfRational}.

\begin{definition}[Static reducible]
\label{def:statics}
We say that a marked cover $(f,A)$ with $|A| = 4$ is \emph{static reducible} if its static portrait has a component consisting of a single directed edge between marked points and that edge has label 1.
\end{definition}

For concreteness, we reformulate this definition in coordinates. Let $(f,A)$ be  a marked branched cover with $A=\{0,1,\infty,a\}$. Then $(f,A)$ is \emph{static reducible} if up to relabeling, $f^{-1}(f(a))$ contains no critical point and no other marked point beside $a$.  In this case, we say that the point $a$ is {\it statically trivial}. As a notational convenience, we will always assume that the point $a$ is statically trivial if $(f,A)$ is statically reducible.

 %The marked point in the domain is called {\it statically trivial}.

\begin{proposition}[Static implies M\"obius reducible]\label{prop:cleanImpliesReducible}
Let $(f,A)$ be a static reducible rational marked cover. Then $(f,A)$ is M\"obius reducible.
\end{proposition}

\begin{proof}
If $f(a)\neq a$, let $\rho$ be the permutation of $A$ that interchanges $a$ and $f(a)$ and consists of two disjoint two-cycles. If $f(a)=a$, let $\rho$ be the identity permutation of $A$. Then $M$ is taken to be the unique M\"obius transposition that induces $\rho$ on $A$. Let $g:=M\circ f$. Evidently $f=M\circ g$ since $M=M^{-1}$. The pair $(g,A)$ is immediately seen to be a marked cover, and $g(a)=a$ where $a$ is statically trivial. Taking $B:=A\setminus\{a\}$, it follows that $(g,B)$ is a marked cover.
  
\end{proof}

\subsection{Maps on moduli space}\label{subsec:mapsOnModuli}
Recall that $A=\{0,1,\infty,a\}\subset\rs$. We identify $\mathcal{M}_A$ and $\rs\setminus\{0,1,\infty\}$ via $\iota\mapsto \iota(a)$, where $\iota$ is normalized so that it maps $0,1,\infty$ to $0,1,\infty$ respectively.

\begin{theorem}[Map on moduli space]\label{Thm:mapOnModuliSpace}
Let $(f,A)$ be a rational marked branched cover that is M\"obius reducible. Then the following diagram commutes (with $g$ and $a$ as in Definition \ref{defn_MobReducible}):

\centerline{ \xymatrix{\mathcal{T}_{A} \ar[r]^{\sigma_{f,A}} \ar[d]_\pi &{\mathcal{T}_A} \ar[d]^{\pi} \\
\mathcal{M}_A  &{\mathcal{M}_A}\ar[l]^{g}}}
\noindent Moreover, the unique fixed point of $\sigma_{f,A}$ is an element of $\pi^{-1}(a)$.
\end{theorem}

\begin{remark}
It is a theorem of \cite{K} that the pullback on Teichm\"uller space associated to an arbitrary Thurston map covers a \emph{correspondence} on moduli space.  It is thus a special feature of our restricted class of maps that the pullback map always covers a \emph{map} on moduli space. In this case, the unique attracting fixed point of $\sigma_f$ in Teichm\"uller space evidently projects to a repelling fixed point of $g$.
\end{remark}

\begin{proof}
Since $M$ is an automorphism of $(\rs,A)$, it follows that $\sigma_{f,A}=\sigma_{M\circ g,A}=\sigma_{g,A}$. Point pushing defines a bijection between the space $\mathcal{T}_A$ and the set of homotopy (rel endpoint) classes of arcs in $\mathbb{C}\setminus\{0,1\}$ emanating from $a$ (\cite[Theorem 2.4]{Lod}). Since $a$ is not critical, there is a unique lift of the path corresponding to $\tau$ under $g$ based at $a$. The homotopy class of this path corresponds to $\sigma_{g,A}(\tau)$ by point pushing. Since the projection of $\sigma_{g,A}$ to $\mathcal{M}_A$ coincides with the branch of $g^{-1}$ that fixes $a$, the conclusion follows from the identity principle for holomorphic maps.

\end{proof}

Notice that the result does not require any restrictions on the degree of the cover $f$. It does however use the fact that the marked set is small in order to guarantee that $M(A)=A$.

\section{Specialization to the quadratic case}
\label{sec:RamPort}

Denote by $Q(i)$ the set of quadratic Thurston maps with exactly $i$ postcritical points. Clearly $Q(1)$ is empty.  Every element of $Q(2)\cup Q(3)$ is combinatorially equivalent to a rational map, and an exhaustive list of rational representatives is given in Table \ref{tbl:Quad3}. 
Our main focus will be the set of non-Euclidean quadratic Thurston maps with four postcritical points, denoted $Q(4)^{*}$. We will show (Lemma \ref{lem:cleanQuadratics}) that $Q(4)^*$ maps are static reducible, allowing us to apply the theory of the previous section.
For more on the Euclidean case, see \cite{M03} and \cite{PET}.
%Rational Euclidean maps in $Q(4)$ are Latt\`es maps, and have been classified up to conjugacy \cite{M03}. %(there are four conjugacy classes). 
%However, there are infinitely many combinatorial classes of Euclidean Thurston maps in $Q(4)$ not equivalent to rational maps.
%Since $Q(4)$ maps are NET maps (all critical points simple and exactly four postrcritical points) and the study of combinatorial classes of such maps can be reduced to the study of similarity classes of matrices, the question of enumeration of classes becomes a question of the number of distinct eigenvalues for a given trace \cite{PET}.
%Thus, the classification of Euclidean maps in $Q(4)$ amounts to solving a class number problem, a notoriously difficult type of problem in number theory that we prefer to avoid \cite{LM33,S07}. 
%Thus 

\renewcommand{\arraystretch}{2.4}
\begin{table}
\caption{All quadratic rational maps with less than four postcritical points}
\begin{tabular}{c|l|l}

\hline

$f\in Q(2)\cup Q(3)$ & dynamical portrait \qquad\qquad\qquad\qquad\qquad  & fixed points \\\hline\hline

%$z^2$  & $\xymatrix{0\ar@{=>}@/^/[r] & 1 \ar@/^/[l]   &\infty\ar@{=>}@/^/[r] &\bullet\ar@/^/[l]}$  & $\bullet = [1]_\mathbb{Z}$\\
%\hline\hline

$(1-2z)^2$& $\xymatrix{
\frac{1}{2}\ar@{=>}[r] & 0 \ar[r]   & 1 \ar@(ur,dr) & \infty\ar@{=>}@(ur,dr)[]}$ &  $\frac{1}{4},1,\infty$ \\
\hline

$\dfrac{1}{1-(1-2z)^2}$& $\xymatrix{
\frac{1}{2}\ar@{=>}[r] & 1 \ar[r]   & \infty \ar@{=>}[r] & 0 \ar@/^/[l] }$ &  $\approx -.4196, .7098\pm .3031i$\\
\hline

$1-\dfrac{1}{(1-2z)^2}$& $\xymatrix{
\frac{1}{2}\ar@{=>}[r] & \infty \ar@{=>}[r]   & 1 \ar[r]  & 0 \ar@(ur,dr)}$ &  $0,1\pm\frac{1}{2}i$ \\
\hline\hline
 %____________________________________________

$z^2$& $\xymatrix{
0 \ar@{=>}@(ur,dr)[]    & \infty\ar@{=>}@(ur,dr)[]}$ &  $ 0,1,\infty$ \\
\hline

$1-z^2$& $\xymatrix{
0 \ar@{=>}[r] & 1 \ar@/^/[l]   & \infty\ar@{=>}@(ur,dr)[]}$ &  $ \frac{1}{2}(-1\pm\sqrt{5}),\infty$ \\
\hline

$\dfrac{1}{z^2}$& $\xymatrix{
0 \ar@/^/@{=>}[r]    & \infty \ar@/^/@{=>}[l]}$ &  $ 1,\frac{1}{2}(-1\pm\sqrt{3}i)$ \\
\hline

$\dfrac{1}{1-z^2}$& $\xymatrix{\infty \ar@{=>}[r] & 0 \ar@{=>}[r] & 1 \ar@/^/[ll] }$ &  $\approx -1.347, 0.6624\pm .5623i$ \\ \hline

\end{tabular}

\caption*{Each rational map in $Q(2)\cup Q(3)$ is listed in the first column (up to automorphism) together with its dynamical portraits and fixed points. The first three maps have one critical postcritical point, and the last four maps have two.}\label{tbl:Quad3}
\end{table}

\subsection{Portrait enumeration}

%A Thurston map $f$ is called \emph{nearly Euclidean} if $|P_f|=4$ and all critical points are simple; we say that $f$ is a \emph{NET map}.  It is known that $f$ is a perturbation of an origami map from an affine two sphere to itself (\cite[Thm 2.1]{CFPP}), a fact which makes lifting multicurves particularly tractable.  Nevertheless, NET maps exhibit a surprising variety of dynamical behaviors far beyond the much simpler origami maps \cite{FKKLPPS}, even in the case of quadratic NET maps which are the object of the present study (we refer to such maps as \emph{QNET maps}). It should be observed that every quadratic Thurston map with exactly four postcritical points is in fact an NET map.

Enumeration of static and dynamical portraits is a sensible first step in the enumeration of combinatorial classes. 

There are exactly two static portraits realized by maps in $Q(4)^*$. These are shown in Figure \ref{fig:statics}. The left-hand static portrait is realized by the first nine $Q(4)^*$ maps in Table \ref{tbl:Hurwitz1} and the right-hand static portrait is realized by the last three $Q(4)^*$ maps in Table \ref{tbl:Hurwitz1}.

\begin{figure}[h]
$$\xymatrix{\ast \ar@{=>}[d] &\bullet \ar@{=>}[d] & \bullet \ar[d] & {\bullet} \ar[d] & \bullet \ar[dl] \\
\bullet &\bullet &\bullet &\bullet}
\qquad \qquad\qquad\qquad \qquad
\xymatrix{\bullet \ar@{=>}[d] &\bullet \ar@{=>}[d] & \bullet \ar[d] & {\bullet} \ar[d] \\
\bullet &\bullet &\bullet &\bullet}$$
\caption{The two possible static portraits for Thurston maps in $Q(4)^*$. The marked (i.e. postcritical) points are labeled by $\bullet$ and the unmarked (i.e. nonpostcritical) critical points are labeled by $\ast$. The left portrait is for those $Q(4)^*$ maps with one critical marked point, and the right portrait is for those with two critical marked points.}
\label{fig:statics}
\end{figure}

One might enumerate the dynamical portraits for maps in $Q(4)^*$ by brute force using the fact that each map has exactly two critical points and each point has two preimages counting multiplicity. This method generates a list of 13 candidates, but there is no guarantee that each dynamical portrait is actually realized by a Thurston map. One way to prove realizibility of a dynamical portrait is to explicitly construct the Thurston map, but this is cumbersome and there are no algorithms in general.

We prefer a more conceptual approach to enumerating portraits. The main result of this section is Proposition \ref{prop:RationalQ4Comps} which asserts the M\"obius reducibility of all rational maps in $Q(4)^*$. One application of this result is the enumeration in Table \ref{tbl:Hurwitz1} of all $Q(4)^*$ dynamical portrait equivalence classes (realized by rational maps) simply by M\"obius postcomposing maps in $Q(2)\cup Q(3)$. Realizability of these dynamical portraits is then automatic and Theorem \ref{Thm:mapOnModuliSpace} enables us to efficiently find maps on moduli space. There is only one (equivalence class of) portrait in $Q(4)^*$ that is not realized by a rational map 
\[\xymatrix{0\ar@{=>}@/^/[r] & 1 \ar@/^/[l]   &\infty\ar@{=>}@/^/[r] &\bullet\ar@/^/[l]}\] but is realized for example as the formal mating of $z^2-1$ with itself. Thus all 13 dynamical portraits are realized by quadratic Thurston maps.

Three portraits in $Q(4)^*$ are realized by polynomials, and though the corresponding Thurston classes are enumerated in \cite{BN06}, we include them for completeness.

\subsection{Quadratics are M\"obius reducible}

\begin{lemma}[$Q(4)^*$ maps are static reducible]\label{lem:cleanQuadratics}
For each $f\in Q(4)^*$, the number of statically trivial postcritical points is equal to $|P_f\cap C_f|\geq 1$.
\end{lemma}
\begin{proof}
Since $|P_f|=4$ and $f$ is quadratic, the set $f^{-1}(P_f)$ contains 8 points counting multiplicity. Clearly $|P_f\cap C_f|+4$ of the preimages are contained in the postcritical set and $2(2-|P_f\cap C_f|)$ preimages lie in critical points outside of the postcritical set. Thus there are 
\[(|P_f\cap C_f|+4)+2(2-|P_f\cap C_f|)\]
preimages that lie in $P_f\cup C_f$, and so there are $|P_f\cap C_f|$ distinct points in the complement. Let $z_0$ be such a point. Then the unique preimage of $f(z_0)$ in $P_f$ is statically trivial.
\end{proof}
This lemma confirms a visual inspection of Figure \ref{fig:statics}.

For some point  $\bullet\in\rs\setminus\{0,1,\infty\}$, denote by $M_{\bullet,0}, M_{\bullet,1},$ and $M_{\bullet,\infty}$ respectively the M\"obius transformations inducing the following permutations: \[ (\bullet\; 0)(1\;\infty),\qquad (\bullet\; 1)(0\;\infty),\qquad (\bullet \;\infty)(0\;1).\]

Recall that the maps in $Q(2)\cup Q(3)$ are already enumerated Table \ref{tbl:Quad3}. The following statement is the basis of our enumeration of rational $Q(4)^*$ maps given in Table \ref{tbl:Hurwitz1}.

\begin{proposition}[Rational $Q(4)^*$ maps are compositions]\label{prop:RationalQ4Comps} Let $f\in Q(4)^*$ be a rational map; fix a normalization so that $P_f=\{0,1,\infty,\bullet\}$ where $\bullet$ is statically trivial. Then there exists $g\in Q(2)\cup Q(3)$ so that $f=M_{\bullet, f(\bullet)}\circ g$.
\end{proposition}

\begin{proof}
Since $\bullet$ is statically trivial, Proposition \ref{prop:cleanImpliesReducible} implies that $f$ may be expressed as the desired composition. 
\end{proof}

The following is a direct consequence of this proposition and Theorem \ref{Thm:mapOnModuliSpace}.

\begin{corollary}
The map $g$ from the M\"obius reduction in Proposition \ref{prop:RationalQ4Comps} is the map on moduli space associated to $\sigma_f:\mathcal{T}_f\to\mathcal{T}_f$ in the sense of Theorem \ref{Thm:mapOnModuliSpace}, and $g(\bullet)=\bullet$.
\end{corollary}

\begin{table}
\caption{All non-Euclidean quadratic rational maps with four postcritical points}
\begin{tabular}{c|c|c|l}
  
\hline
%Headings
$f\in Q(4)^*$  & $\bullet$ & $g$ &  dynamical portrait of $f$ \\\hline\hline

%First block (Hurwitz 1)

$M_{\bullet,0}\circ g$ & $ \frac{1}{4}$&  & $\xymatrix{
\frac{1}{2}\ar@{=>}[r] &\bullet\ar[r] & 0\ar[r] & \infty \ar@{=>}@/^/[r] & 1 \ar@/^/[l]}$   \\ 
  \cline{1-2}\cline{4-4}%\hline\hline

$M_{\bullet,1}\circ g$ & $ \frac{1}{4}$ & $(1-2z)^2$  & $\xymatrix{\frac{1}{2}\ar@{=>}[r] &\infty\ar@{=>}[r] & 0 \ar[r] &\bullet\ar@/^/[r] & 1 \ar@/^/[l]}$   \\ \cline{1-2}\cline{4-4}

 $M_{\bullet,\infty}\circ g$ & $ \frac{1}{4}$& &  $\xymatrix{\frac{1}{2}\ar@{=>}[r] & 1 \ar[r] & 0 \ar@(ur,dr)[] &\infty\ar@{=>}@/^/[r] & \bullet \ar@/^/[l]}$    \\ \hline

%Second block (Hurwitz 1)

$M_{\bullet,1}\circ g$ & $ \approx -.4196, .7098\pm .3031i$  &  & $\xymatrix{\frac{1}{2}\ar@{=>}[r] &\bullet\ar[r] & 1 \ar[r] & 0 \ar@(ur,dr)[] &\infty \ar@{=>}@(ur,dr)[]}$  \\ \cline{1-2}\cline{4-4}
  
$M_{\bullet,\infty}\circ g$ &  $ \approx -.4196, .7098\pm .3031i$& $\dfrac{1}{1-(1-2z)^2}$  & $\xymatrix{
\frac{1}{2}\ar@{=>}[r] & 0 \ar[r] &\bullet\ar[r] &\infty \ar@{=>}[r] & 1 \ar@/^/[ll]}$  \\\cline{1-2}\cline{4-4}

$M_{\bullet,0}\circ g$ &  $ \approx -.4196, .7098\pm .3031i$ &  & $\xymatrix{
 \frac{1}{2} \ar@{=>}[r] &\infty\ar@{=>}[r] &\bullet\ar[r] & 0 \ar[r] & 1 \ar@(ur,dr)[] }$ \\ 
  \hline

%Third block (Hurwitz 1)

$M_{\bullet,\infty}\circ g$&   $  1\pm \frac{1}{2}i$ &  & $\xymatrix{\frac{1}{2}\ar@{=>}[r] &\bullet\ar[r] &\infty\ar@{=>}[r] & 0 \ar[r] & 1 \ar@(ur,dr)[]}$  \\   \cline{1-2}\cline{4-4}

$M_{\bullet,0}\circ g$ &  $ 1\pm \frac{1}{2}i$ & $1-\dfrac{1}{(1-2z)^2}$  & $\xymatrix{\frac{1}{2}\ar@{=>}[r] & 1\ar[r] & \bullet \ar@/^/[r] & 0 \ar@/^/[l] &\infty\ar@{=>}@(ur,dr)[]}$  \\ \cline{1-2}\cline{4-4}

$M_{\bullet,1}\circ g$ & $ 1\pm \frac{1}{2}i$ &   & $\xymatrix{\frac{1}{2}\ar@{=>}[r] & 0 \ar[r] &\infty \ar@{=>}[r] &\bullet\ar[r] & 1 \ar@/^/[ll]}$     \\  \hline\hline

%First block (Hurwitz 2)  
  
$M_{\bullet,1}\circ g$ &  $ \frac{1}{2}(-1\pm\sqrt{5})$ & $1-z^2$ & $\xymatrix{
\infty\ar@{=>}[r] & 0 \ar@{=>}[r]   &\bullet\ar[r] & 1 \ar@/^/[lll]}$ \\
\hline

%Second block (Hurwitz 2)

$M_{\bullet,0}\circ g$ &  $\frac{1}{2}(-1\pm\sqrt{3}i)$ & $\frac{1}{z^2}$ & $\xymatrix{
0 \ar@{=>}[r] & 1 \ar[r] &\infty \ar@{=>}[r] &\bullet \ar@/^/[lll]}$  \\
\hline

%Fourth block (Hurwitz 2)

$M_{\bullet,1}\circ g$ &  $\approx -1.347, 0.6624\pm .5623i$ & $\frac{1}{1-z^2}$ & $\xymatrix{0 \ar@{=>}[r] &\bullet\ar[r] & 1 \ar@/^/[ll] &\infty\ar@{=>}@(ur,dr)[]}$  \\ \hline

\end{tabular}
\caption*{Each rational map $f\in Q(4)^*$ is listed in the first column (up to M\"obius conjugacy) as a composition of a simpler map $g\in Q(2)\cup Q(3)$ and a M\"obius transformation depending on $\bullet\in\mathbb{C}$ (guaranteed by Proposition \ref{prop:RationalQ4Comps}). The first nine maps satisfy $|P_f\cap C_f|=1$ and the last three maps satisfy $|P_f\cap C_f|=2$. Double arrows indicate nontrivial ramification.}\label{tbl:Hurwitz1}
\end{table}

\begin{remark}[Group action on portraits]
Dynamical portraits in Table \ref{tbl:Hurwitz1} admit symmetries that are typically not induced by conformal automorphisms. For example, the permutation $( 0 \;1\; \infty)$ acts by postcomposition on the first 9 portraits as follows: the $i$-th portrait is taken to the $(i+3)$-th portrait (mod 9).
\end{remark}

%\section{M\"obius reducibility beyond quadratics}
%\label{sec:StatPort}

\subsection{Hurwitz equivalence}

In the quadratic case, a classical result allows us to use the solution to the twisting problem to enumerate combinatorial equivalence classes of Thurston maps.

\begin{definition}[Hurwitz equivalences]
\label{Hurwitzdef}
Two Thurston maps $f$ and $g$ are said to be (classically) \emph{Hurwitz equivalent} if and only if there exist orientation-preserving homeomorphisms $h_0, h_1:\sphere\to\sphere$ so that
\[h_0\circ f= g\circ h_1\]
We say that $f$ and $g$ are {\it pure Hurwitz equivalent} if $h_0$ and $h_1$ agree on $P_f$, and their shared image is in $P_g$.
\end{definition}

\begin{lemma}[Pure Hurwitz classes correspond to dynamical portraits]
\label{Hurwitzlem}
For $f,g \in Q(4)^*$, $f$ and $g$ are pure Hurwitz equivalent if and only if they have equivalent dynamical portraits.
\end{lemma}

\begin{proof}
The dynamical portrait is always an invariant for the pure Hurwitz class of a Thurston map.
We need to show that for $Q(4)^*$ maps, every pair of maps that have equivalent dynamical portraits are also pure Hurwitz equivalent.

By the uniqueness theorem for simple branched covers of $\sphere$ (see e.g. \cite[Theorem 3.4]{BE}), any two maps in $Q(4)^*$ are classically Hurwitz equivalent. If two such maps have equivalent dynamical portraits, then the maps $h_0, h_1$ in Definition \ref{Hurwitzdef} must in fact satisfy the definition for pure Hurwitz equivalence.

\end{proof}

\begin{remark}
\label{impurerem}
One can also define the impure Hurwitz class of a Thurston map by allowing $h_0$ and $h_1$ to disagree on $P_f$, but still send $P_f$ into $P_g$. In our $Q(4)^*$ setting, the impure Hurwitz class corresponds to static portrait equivalence.
\end{remark}
%
%\rbl{ Show that in our setting dynamical portrait is an invariant of pure Hurwitz class, cf \cite[\S 5]{FKKLPPS}}
% 

%Then Proposition \ref{prop:cleanImpliesReducible} and Theorem \ref{Thm:mapOnModuliSpace} yield the following:

%\begin{theorem}[Static reducible implies map on moduli space]
%\label{thm:static}
%If $(f,A)$ is a rational marked cover that is static reducible, then $f$ is M\"obius reducible as in Definition \ref{defn_MobReducible} and the following commutes 

%\centerline{ \xymatrix{\mathcal{T}_{A} \ar[r]^{\sigma_{f,A}} \ar[d]_\pi &{\mathcal{T}_A} \ar[d]^{\pi} \\
%\mathcal{M}_A  &{\mathcal{M}_A}\ar[l]^{g}}}
%\end{theorem}

% \rbl{Our Prop 3 gives a sufficient condition in any degree to determine whether the pullback map covers a map on moduli space...this has only been for polynomials by Sarah and maybe quadratic rational maps by Wolf Jung (see draft of Quadratic matings and Lattès maps on his website)}

\section{Background on self-similar groups}
\label{sec:ssgrps}
%\subsection{General theory} 

A self-similar group is understood through its action by isomorphisms on a tree.
Let $X = \{1, 2, \dots, d\}$ be a finite alphabet.
We identify the set $X^*$ of finite words in $X$ with the vertices of an infinite, $d+1$-regular, rooted tree.
We will think of the tree as ``growing'' up from the root, so the children of each vertex will be ``above'' it.
Note that an isomorphism of this tree will preserve the ``levels'' of the tree
(i.e. the length of words in $X^*$).
Further, the isomorphism will take the subtree above a vertex to the subtree above the image of that vertex (i.e. words with a given suffix will be taken to words that end with the image of that suffix).
Since the subtree above any vertex of an infinite, $d+1$-regular, rooted tree is canonically isomorphic to the full tree,
we can identify the subtrees above the pre-image and image vertices with the full tree.
Thus, the \emph{restriction} of the isomorphism to the subtree above a particular vertex can be viewed as an element of the isomorphism group of the full tree.

A group action by isomorphisms on an infinite, $d+1$-regular, rooted tree is \emph{self-similar} 
if it is closed under restrictions.
That is, the restriction of any element of the group to any vertex will be an element of the group.

An element of a self-similar group (i.e. a group admitting a self-similar action) can be described by
giving the action of the element on the first level of the tree and the restrictions of the element to all vertices on that first level.
This notation is called the \emph{wreath recursion} of the element.
For example, for an element $h$ of a group with a self-similar action on a binary rooted tree, we would write
$$h = \left< h|_1, h|_2\right>\pi_h$$
where $h|_1$ represents the restriction of $h$ to the vertex identified with $1$, $h|_2$ the restriction of $h$ to the vertex $2$, and $\pi_h$ is the permutation on $\{1,2\}$ induced by $h$.
We omit $\pi_h$ if it is trivial, and we denote the non-trivial permutation on $\{1,2\}$ by $\sigma$.

To multiply wreath recursions, use the following rule:
$$\left< g|_1, g|_2, \dots , g|_d\right>\pi \cdot \left<h|_1, h|_2, \dots , h|_d\right>\tau = \left< g|_1h|_{\pi(1)}, g|_2h|_{\pi(2)}, \dots , g|_dh|_{\pi(d)}\right>\pi\tau$$
See section 1.4 of \cite{N} for a discussion of the connection between self-similar groups and wreath products.

A \emph{virtual endomorphism} of a group $G$ is a homomorphism from a finite-index subgroup of $G$ to $G$.
Given a self-similar group $G$ and a letter $x$ of the tree with vertex set $X^*$ that it acts on, 
the restriction map at $x$ is a natural virtual endomorphism from the stabilizer of $x$ to $G$.
So, for example, suppose $G$ acts on an infinite, binary, rooted tree whose verices we identify with $\{1,2\}^*$.
Then any element of $G$ can be written as $h = \langle h|_1, h|_2\rangle\pi_h$.
The virtual endomorphism associated with the first coordinate has domain equal the set of elements
$h\in G$ such that $\pi_h$ is trivial and such an element is mapped to $h|_1$.

Given a virtual endomorphism $\phi$ of a group $G$, 
one can construct a self-similar action of $G$ on a tree with vertex set $X^*$, 
where $X$ is the set of cosets of the domain of $\phi$.
For details see section 2.5 of \cite{N}.

A self-similar action of $G$ on a tree $X^*$
is said to be \emph{contracting} if there exists a finite set $\mathcal{N} \subset G$ such that
for every $h\in G$ there exists an $n_h\in\N$ such that 
$h|_v \in \mathcal{N}$ for all words $v\in X^*$ with length at least $n_h$.
The minimal such set $\mathcal{N}$ is unique and is called the \emph{nucleus} of the self-similar action.
The self-similar action is contracting if and only if the \emph{contracting coefficient}
$$\rho = \limsup_{n\to\infty}\sqrt[n]{\limsup_{l(h)\to\infty} \max_{v\in X^n} \frac{l(h|_v)}{l(h)}}$$
is strictly less than 1.
One should think of a contracting self-similar action as one where the restrictions of a group element get
simpler (i.e. shorter in word length) as you restrict further up the tree.
At some height of the tree, all of these restrictions will be in a finite set (the nucleus).

It is important to note that the nucleus of a contracting self-similar action is defined to be a finite set.
In this work, we will be making use of a set with the same definition as the nucleus
%(i.e. the forward orbit of the attractor of the dynamical system of $G$ with its restriction maps),
but we will not require the set to be finite.
We will refer to this set as the \emph{(possibly infinite) nucleus}.

We will briefly describe another formulation of the self-similar theory.
Given a self-similar action of a group $G$ on a tree $X^*$, we can define a \emph{$G$-biset} 
(sometimes called a \emph{bimodule}) as follows:
the set itself is $G\times X$, with elements written as formal products $h\cdot x$.
There is a left action of $G$ by left multiplication, and a right action of $G$ by
$$(g\cdot x) \cdot h = gh|_x \cdot h(x)$$

If the associated self-similar action is contracting, we say that the biset is \emph{hyperbolic}.
If the action is not contracting, but the action by the faithful quotient is, 
then we say that the biset is \emph{sub-hyperbolic}.
In other words, in the sub-hyperbolic case the infinite nucleus contains only finitely-many distinct actions.
All of the bisets for the actions we discuss in this paper are either hyperbolic (when the nucleus is finite)
or sub-hyperbolic (when the nucleus is infinite).
See Table \ref{tbl:Nuclei} and Theorem \ref{thm:subhyp} for details.

In \cite{BN06}, Bartholdi and Nekrashevych use this self-similar machinery to solve the twisting problem for all polynomials in $Q(4)^*$. Specifically, they take a virtual endomorphism $\phi$ of the pure mapping class group $G$
and extend it to a map (not a homomorphism) $\bar{\phi}:G\to G$
such that for $T\in G$, $T\circ f$ is equivalent to $\bar{\phi}(T)\circ f$.
Thus, the question of equivalence class of the twists can be reduced to those twists that are periodic under $\bar{\phi}$.
Since the $G$-biset of the self-similar action associated with $\phi$ is sub-hyperbolic in these cases,
this set is relatively straightforward to analyze.

\section{Wreath recursions for maps on moduli space}\label{sec:wreath}

We first represent pure mapping classes of $(\mathbb{S}^2,P_f)$ as elements of the fundamental group of moduli space. This induces an isomorphism between the mapping class biset of $f$ and the fundamental group biset of the map on moduli space at an appropriate basepoint according to \cite{BN06} and \cite[Theorem 2.6]{Lod}. If $f$ is rational, this basepoint is the fixed point associated with $f$. Wreath recursions for all maps on moduli space at all fixed points are computed, and their nuclei are computed in a generalized sense.  These nuclei will be vital for both our solution to the twisting problem and our computation of the global dynamics of the pullback on curves.

\subsection{Twisting in terms of fundamental group}
The point-pushing isomorphism is used here to write the left action of the pure mapping class group as a right action of the fundamental group of moduli space.  Let $\mathcal{M}_f$ denote moduli space (identified with $\mathbb{C}\setminus\{0,1\}$ as before) and suppose $z_0$ is the fixed point of the map on moduli space $g$ corresponding to some map $f\in Q(4)^*$. Let $\gamma:[0,1]\rightarrow\mathcal{M}_f$ be a loop with $\gamma(0)=\gamma(1)=z_0$.  Then $\gamma$ extends to a motion $\tilde{\gamma}:\mathcal{M}_f\times[0,1]\rightarrow\mathcal{M}_f$ that ``pushes" $z_0$ along $\gamma$; in other words $\tilde{\gamma}(\cdot,t)$ is a homeomorphism and $\tilde{\gamma}(z_0,t)=\gamma(t)$ for all $t$. Then to $[\gamma]\in\pi_1(\mathcal{M}_f,z_0)$ we associate the isotopy class of $T_{\gamma}:=\tilde{\gamma}(\cdot,1)$. This is a pure mapping class of $(\mathbb{S}^2, P_f)$ under the identification of $\rs$ with $\mathbb{S}^2$ via the stereographic projection. We use the notation $f\cdot \gamma:=T_{\gamma}\circ f$, following \cite[\S 8]{Lod} and \cite{BN06}.

\subsection{Generators and connecting paths}
We need to compute a wreath recursion for each map on moduli space based at each fixed point outside of $\{0,1,\infty\}$ (suitably interpreted for $g(z)=z^2$, which has none). Such wreath recursions depend on non-canonical choices of generators and connecting paths; in the interest of streamlining our presentation we fix generators $\hat{\alpha},\hat{\beta}$ in Figure \ref{fig:generators} and then use the M\"obius action to produce generators for all other maps on moduli space (labelled $\alpha$ and $\beta$).  Connecting paths will be chosen so as to simplify the wreath recursion presentations, and in particular the first coordinate of every active element will be trivial (i.e. $\alpha|_1 =1$). This fact will simplify the computations in the proofs of Lemmas \ref{NaNlem}, \ref{phihatdef}, and \ref{phihatNlem}. All wreath recursions and their nuclei are recorded in Table \ref{tbl:Nuclei}.

\begin{figure}[h]
\centerline{\includegraphics[width=70mm]{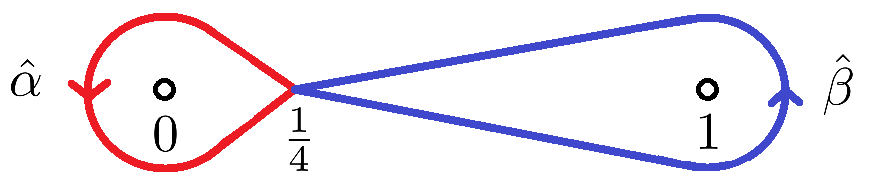}}
\caption{Generators $\hat{\alpha}$ and $\hat{\beta}$ of $\pi_1(\mathbb{C}\setminus\{0,1\},1/4)$}
\label{fig:generators}
\end{figure}

\subsubsection{One critical postcritical point}
 The map $g_1(z)=(1-2z)^2$ has exactly one non-postcritical fixed point, namely $1/4$.  The first three maps on moduli space that arise in Table \ref{tbl:Hurwitz1}  have the form $M\circ g_1$ where $M$ is the unique M\"obius transformation inducing the permutations $id, (0\; 1\; \infty)$ and $(0\; \infty\; 1)$ respectively.

\begin{figure}[h]
\centerline{\includegraphics[width=70mm]{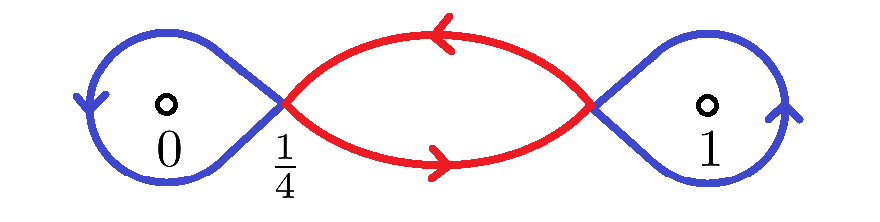}}
\caption{Lifts of generators from Figure \ref{fig:generators} under $g_1$}
\label{fig:LiftHurwitz1}
\end{figure}

Now let $z_0$ be some fixed point of $M\circ g_1$ outside $\{0,1,\infty\}$.  If $z_0$ can be connected to $M(1/4)$ by a line segment in $\mathbb{C}$ that doesn't pass through $0$ or $1$, let $\ell$ be this line segment (clearly such line segments exist when $z_0\notin \mathbb{R}$).  Otherwise, fix a path in $\mathbb{R}\cup\{\infty\}$ connecting $z_0$ to $M(1/4)$ that passes through exactly one point in $\{0,1,\infty\}$ which we denote $p$; then $\ell$ is taken to be this path perturbed to the left of $p$. In either case, we fix $\alpha:=\ell\cdot M(\hat{\alpha})\cdot\bar{\ell}$ and $\beta:=\ell\cdot M(\hat{\beta})\cdot\bar{\ell}$ as the generators of $\pi_1(\mathbb{C}\setminus \{0,1\}, z_0)$, where $\bar{\ell}$ denotes the reverse of $\ell$.

For the wreath recursion computation, the label $1$ is assigned to the fixed point $z_0$ and the label $2$ is assigned to the other preimage of $z_0$ under $M\circ g_1$.  The connecting paths for $M\circ g_1$ are taken to be the constant path at $z_0$ and the unique lift of $\alpha$ under $M\circ g_1$ based at $z_0$. %\[(M\circ g_1)^{-1}(\ell\alpha\bar{\ell})[z_0].\]

\subsubsection{Two critical postcritical points}

The last three maps on moduli space in Table \ref{tbl:Hurwitz1} have the form $M\circ g_2$ where $M$ is the M\"obius transformation inducing the permutations $(0\; 1)$, $(0 \;\infty)$, and $(0\; 1\; \infty)$ respectively.  Each of these three maps has a fixed point outside $\{0,1,\infty\}$. The construction of generators $\alpha,\beta$ and wreath recursions is carried out as in the case of one critical postcritical point (i.e. apply the M\"obius action to $\hat{\alpha}$ and $\hat{\beta}$ and choose $\ell$ as before to compute $\alpha,\beta$ and the wreath recursion).

We must separately construct generators and connecting paths for $g_2(z)=z^2$ because of its exceptional nature (specifically $g_2$ does not have any fixed points outside of $\{0,1,\infty\}$). For lack of a natural basepoint, we simply fix $1/4$ to be the basepoint of the fundamental group and let $\alpha:=\hat{\alpha}$ and $\beta:=\hat{\beta}$ be the generators. The connecting paths for $g_2$ are chosen as follows: $\ell_{1}$ is the path in $\mathbb{R}$ connecting $1/4$ to $1/2$ and $\ell_{2}$ is the path connecting $1/4$ to $-1/2$ along a semi-circular arc in the upper half-plane.

\begin{figure}[h]
\centerline{\includegraphics[width=70mm]{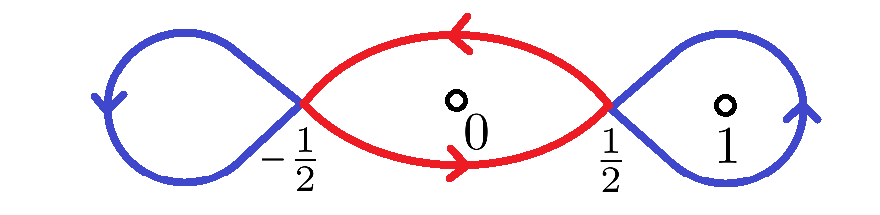}}
\caption{Lifts of generators from Figure \ref{fig:generators} under $g_2$}
\label{fig:LiftHurwitz2}
\end{figure}

\subsection{Nuclei}
Having determined the wreath recursions, 
we compute each (possibly infinite) nucleus using a technique described in Lemma 2.11.2 in \cite{N}.
Specifically, we create a candidate nucleus $N$ 
by closing a generating set under the operations of restriction and inversion (such a set is {\it state-closed}). 
Usually we use $\{1, \alpha, \beta\}$ for this set, but sometimes $\beta$ will not be in the nucleus.
In this case, we use $\{1, \alpha, \gamma\}$ where $\gamma = \beta^{-1}\alpha^{-1}$ or $\{1,\alpha,\delta\}$ where $\delta = \alpha^{-1}\beta^{-1}$.

We then check every element of $N^2$ to determine if there is a number $n$ such that
after \emph{any} $n$ restrictions the result is in $N$.
If this property holds for all elements of $N^2$, then $N$ contains our nucleus, 
since we can use $N$ as a generating set and produce contraction in word length.
If this property does not hold, there will be some cycles of elements of $N^2$ under restrictions.
We create a new candidate nucleus by including these cycles in $N$ and we repeat the process.

For example, consider the wreath recursion $\alpha = \langle 1, \gamma\rangle\sigma$, $\beta = \langle 1, \beta\rangle$, which arises from the $g$-map $\frac{1}{z^2}$ with fixed point $\frac{1}{2}(-1-\sqrt{3}i)$. 
Set $N = \{1, \alpha, \beta^n, \gamma\ | \ n\in\Z\}^{\pm1}$, since $\alpha$ restricts to $\gamma$ and all powers of $\beta$ will self-restrict.
When checking products in the set $N^2$, we see that some of them do not restrict back into $N$.
Since $N$ is state-closed, all of the restrictions of $N^2$ do fall in $N^2$, however.
The restrictions of $N^2$ that do {\it not} fall in $N$ are displayed below
(we do not display the symmetric relations on the inverses).
$$\alpha\gamma^{-1} \to \gamma\beta \to \gamma^{-1}\beta \qquad \gamma^{-2}\to\gamma\beta \qquad \beta^n\alpha \to \beta^n\gamma \qquad \gamma\beta^n\to \alpha\beta^{n+1}$$
Since the longest path along these restrictions has length two, 
we can see that every element of $N^2$ will restrict into $N$ after three restrictions.
Thus $N$ is our (infinite) nucleus.

We present the results of these computations in Table \ref{tbl:Nuclei}. 
Again, we point out that we are allowing our ``nucleus'' to be infinite, unlike in the standard literature.

\begin{landscape}
\begin{table}
\begin{tabular}{c|cll}
 % \hline

%Headings
$g$-map & Fixed point & Wreath recursion  & (Possibly infinite) ``Nucleus'' \\\hline\hline

%First block (Hurwitz 1)

 \hline $(1-2z)^2$ & $\frac{1}{4}$ & $\alpha = \left<1,1\right>\sigma, \quad \beta=\left<\alpha,\beta\right>$ & $\{1, \alpha^n, \beta^n \ |\ n\in \Z \}$\\ \cline{2-4}\hline

%Second block (Hurwitz 1)
& -.4196 &  $\alpha = \left<1,1\right>\sigma, \quad \beta=\left<\delta,\beta\alpha\beta^{-1} \right>$ & $\{1, \alpha, \beta, \delta, \beta\alpha\beta^{-1}, \delta^2, \alpha^{-1}\beta^{-2}, \beta\delta\beta^{-1}, \beta\alpha^2\beta\alpha, \delta^{-1}\beta\delta, \delta^{-1}\alpha\beta\delta\}^{\pm1}$ \\ \cline{2-4}
  
$\dfrac{1}{1-(1-2z)^2}$& $.7098+.3031i$  &   $\alpha = \left<1,1\right>\sigma, \quad \beta=\left<\alpha,\gamma \right>$ & $\{1, \alpha, \gamma\}^{\pm1}$\\\cline{2-4}

& $.7098-.3031i$  &   $\alpha = \left<1,1\right>\sigma, \quad \beta=\left<\alpha,\delta \right>$   & $\{1, \alpha, \delta\}^{\pm1}$ \\ 
  \hline

%Third block (Hurwitz 1)

\multirow{2}{*}{$1-\dfrac{1}{(1-2z)^2}$}  & $1+\frac{i}{2}$ &   $\alpha = \left<1,1\right>\sigma, \quad \beta=\left<\gamma,\beta \right>$ & $\{1, \alpha^n, \beta^n, \gamma^n, \beta^{-1}\alpha^n\beta, \alpha^n\beta, \alpha\beta^2\ | \ n\in \Z\}^{\pm 1}$ \\   \cline{2-4}

 & $1-\frac{i}{2}$  & $\alpha = \left<1,1\right>\sigma, \quad \beta=\left<\delta,\beta \right>$ & $\{1, \alpha^n, \beta^n, \delta^n, \beta\alpha^n\beta^{-1}, \beta\alpha^n, \beta^2\alpha\ | \ n\in \Z\}^{\pm 1}$\\ \cline{2-4}\hline%\noalign{\bigskip}\noalign{\medskip}\hline

%First block (Hurwitz 2)

\hline\hline $z^2$ & N/A & $\alpha = \left<1,\alpha\right>\sigma, \quad \beta=\left<\beta,1\right>$ & $\{1, \alpha, \beta^n, \alpha\beta^n, \alpha\beta^n\alpha^{-1}\ |\ n\in\Z\}^{\pm1}$\\ \hline

%Second block (Hurwitz 2)

\multirow{2}{*}{$1-z^2$}  & $\frac{1}{2}(-1+\sqrt{5})$ &   $\alpha = \left<1,\beta \right>\sigma, \quad \beta=\left<\alpha,1 \right>$ & $\{1, \alpha, \beta, \alpha\beta^{-1}\}^{\pm1}$ \\   \cline{2-4}

 & $\frac{1}{2}(-1-\sqrt{5})$  & $\alpha = \left<1,\beta\right>\sigma, \quad \beta=\left<1,\beta\alpha\beta^{-1} \right>$ & $\{1, \alpha, \beta, \delta, \beta\alpha\beta^{-1}\}^{\pm1}$\\ \cline{2-4}\hline

%Third block (Hurwitz 2)

\multirow{2}{*}{$\dfrac{1}{z^2}$}  & $\frac{1}{2}(-1+\sqrt{3}i)$ &   $\alpha = \left<1,\delta \right>\sigma, \quad \beta=\left<1,\alpha^{-1}\beta\alpha \right>$ & $\{1,\alpha, \beta^n, \alpha^{-1}\beta^n\alpha, \beta^n\alpha\ |\ n\in\Z \}^{\pm1}$\\   \cline{2-4}

 & $\frac{1}{2}(-1-\sqrt{3}i)$  & $\alpha = \left<1,\gamma\right>\sigma, \quad \beta=\left<1,\beta \right>$ & $\{1, \alpha, \beta^n,\gamma\ |\ n\in\Z\}^{\pm1}$\\ \cline{2-4}\hline

%Fourth block (Hurwitz 2)

& -1.347 &  $\alpha = \left<1,\delta\right>\sigma, \quad \beta=\left<1,\alpha \right>$ & $\{1,\alpha,\beta,\delta\}^{\pm1}$ \\ \cline{2-4}
  
$\dfrac{1}{1-z^2}$& $0.6624+.5623i$  &   $\alpha = \left<1,\gamma\right>\sigma, \quad \beta=\left<\alpha,1\right>$ & $\{1, \alpha, \beta, \gamma\}^{\pm1}$\\\cline{2-4}

& $0.6624-.5623i$  &   $\alpha = \left<1,\delta \right>\sigma, \quad \beta=\left<\alpha,1 \right>$   & $\{1, \alpha, \beta, \delta\}^{\pm1}$\\ 
  \hline

\end{tabular}
\caption{Wreath recursions for the maps on moduli space at all fixed points outside $\{0,1,\infty\}$, and the (possibly infinite) ``nucleus'' for those recursions, although for brevity we do not list inverses for the nucleus. Recall that $\gamma = \beta^{-1}\alpha^{-1}$ and $\delta = \alpha^{-1}\beta^{-1}$.
}\label{tbl:Nuclei}
\end{table}
\end{landscape}

\begin{theorem}
\label{thm:subhyp} 
The image of each of the (possibly infinite) nuclei listed in Table \ref{tbl:Nuclei} 
is finite in the faithful quotient of the self-similar action.
That is, the bisets of these self-similar actions are all sub-hyperbolic.
\end{theorem}

\begin{proof}
The faithful quotient of the action is the iterated monodromy group of the map $g$.
Recall that $g$ is always a member of $Q(2)\cup Q(3)$, so it is a postcritically finite rational function.
Then by Theorem 6.4.4 of \cite{N}, the image of the nucleus in $IMG(g)$ is finite.

However, this result can also be seen directly from Table \ref{tbl:Nuclei}.
Each element whose entire cyclic subgroup appears in a nucleus in Table \ref{tbl:Nuclei} 
has order either 1, 2, or 4 in the faithful quotient. 
This can be checked using the wreath recursions given in the table 
(for example, any map with one critical postcritical point has $\alpha^2$ acting trivially).
\end{proof}

\section{Solving the twisting problem}\label{sec:TwistedNET}
%You can find the input data for all rational QNET maps at \cite{QNET}.
%\todo{Add an intro here}

Recall that the twisting problem is the determination of the combinatorial class of a map $h \circ f$ where $f$ is a Thurston map and $h$ is a pure mapping class of $(\mathbb{S}^2, P_f)$.  
As above, we will use the identification $f \cdot \gamma := T_\gamma \circ f$ and treat $h$ as an element of the fundamental group of moduli space.

We give an algorithm to solve this problem for all maps $f\in Q(4)^*$ in the manner of \cite{BN06}:
define a virtual endomorphism $\phi$ on $\pi_1(\mathcal{M}_f,z_0)$ (which we identify with the pure mapping class group) and
extend it to a non-homomorphic map $\bar{\phi}$ on the entire group such that 
$ f \cdot h \simeq f \cdot \bar{\phi}(h)$ 
where $\simeq$ denotes combinatorial equivalence
(we will use $=$ to denote homotopy equivalence). Then we find the attractor of forward iteration of $\bar{\phi}$ on the group and 
determine the combinatorial equivalence class of each of those twists on $f$.

Recall that by Lemma \ref{Hurwitzlem}, the twists of $f$ will meet every combinatorial equivalence class that has the same dynamical portrait as $f$.
So we will enumerate all combinatorial equivalence classes so long as we solve the twisting problem for each of the dynamical portraits.

\subsection{The extended virtual endomorphism $\bar{\phi}$}
\label{sec:phibar}

For a map $f\in Q(4)^*$, we define the virtual endomorphism $\phi$ on $\pi_1(\mathcal{M}_f,z_0)$ as the first coordinate virtual endomorphism from the wreath recursion in Table \ref{tbl:Nuclei}.
So for the map in $Q(4)^*$ defined by fixed point $1/4$ of the $g$-map $(1-2z^2)$ for example, 
the virtual endomorphism $\phi$ would have domain generated by $\alpha^2, \beta, \alpha^{-1}\beta\alpha$
and would map the generators by 
\begin{equation}\label{eqn:veexample}
\phi(\alpha^2) = 1, \phi(\beta)=\alpha, \phi(\beta^\alpha) = \beta.
\end{equation}

We then define a map (not a homomorphism) $\phibar$ from the pure mapping class group to itself:
$$\phibar(h) = \left\{\begin{array}{ll}
\phi(h) & \text{if } h \in \text{Dom}(\phi)\\
\alpha\phi(h\alpha^{-1}) & \text{if } h \in \text{Dom}(\phi)\alpha
\end{array}\right.$$

\begin{lemma}

For $h\in\pi_1(\mathcal{M}_f,z_0)$,
$$f\cdot h \simeq f\cdot\phibar(h)$$

\end{lemma}

\begin{proof}
Notice if $h\in\text{Dom}(\phi)$, then $h$ lifts under $f$ to $\phi(h)$:
$$f\cdot h =\phi(h)\cdot f \simeq f\cdot \phi(h) = f\cdot \phibar(h)$$
and if $h\in\text{Dom}(\phi)\alpha$, then
$$f\cdot h = f\cdot h\alpha^{-1}\cdot \alpha = \phi(h\alpha^{-1})\cdot f\cdot \alpha \simeq f\cdot \alpha\phi(h\alpha^{-1})=f\cdot \phibar(h)$$
Thus, we have defined $\phibar$ such that $f\cdot h \simeq f\cdot \phibar(h)$ for all $h\in\pi_1(\mathcal{M}_f,z_0)$.
\end{proof}

To solve the twisting problem, we will study the dynamics of $\phibar$ on $\pi_1(\mathcal{M}_f,z_0)$.

\begin{lemma}
\label{NaNlem}

Let $h\in\pi_1(\mathcal{M}_f,z_0)$. 
There exists $N$ such that for all $n>N$
$$\phibar^{\circ n}(h) \in \mathcal{N}\cup\alpha\mathcal{N}$$

\end{lemma}

\begin{proof}
Recall that we chose to make the wreath recursion of $\alpha$ of the form $\alpha=\left<1, \alpha|_2\right>\sigma$.
Notice that if $h=\left<h_1, h_2\right>$, then $\phibar(h) = h_1$ and
$$\phibar(\alpha h) = \alpha\phi(\alpha h\alpha^{-1})=\alpha\phi(\left<1, \alpha|_2\right>\sigma\left<h_1, h_2\right>\left<\alpha|_2^{-1}, 1\right>\sigma) = \alpha h_2$$
and if $h=\left<h_1, h_2\right>\sigma$, then 
$$\phibar(h) = \alpha\phi(h\alpha^{-1}) = \alpha\phi(\left<h_1,h_2\right>\sigma\left<\alpha|_2^{-1}, 1\right>\sigma) = \alpha h_1$$
and
$$\phibar(\alpha h) = \phi(\alpha h)=\phi(\left<1, \alpha|_2\right>\sigma\left<h_1, h_2\right>\sigma) = h_2$$
Thus, $\phibar(h)$ is either equal to a restriction of $h$ or $\alpha$ times a restriction of $h$.
Since taking repeated restrictions of $h$ eventually enters (and remains in) $\mathcal{N}$, by the computations above we have that iterating $\phibar$ on any element of the fundamental group will eventually land in the set $\mathcal{N}\cup \alpha\mathcal{N}$.
\end{proof}

Consequently, to solve the twisting problem we compute the action of $\phibar$ on the set $\mathcal{N}\cup\alpha\mathcal{N}$. For example, consider the virtual endomorphism defined by equations \ref{eqn:veexample} above.
The resulting dynamics of $\bar{\phi}$ on the set $\mathcal{N}\cup\alpha\mathcal{N}$ are:
$$\beta^{2n}\to\alpha^{2n}\to1\to1\qquad\beta^{2n+1}\to\alpha^{2n+1}\to\alpha\to\alpha\qquad \alpha\beta^n\to\alpha\beta^n$$
%Otherwise, we compute the virtual endomorphism and nucleus at the end of the document.
%
%\gak{ drop this next bit since we're moving the discussion from earlier in the document}
%As described in Section \ref{sec:phibar}, the map (not a homomorphism) $\bar{\phi}$ from the group to itself 
%preserves the conjugacy class of elements.
%That is, $f_0\cdot g \simeq f_0\cdot \bar{\phi}(g)$ for all pure mapping classes $g$.
%Also as described in that section, for any pure mapping class $g$ and for $n$ sufficiently large,
%$\bar{\phi}^{\circ n}(g) \in \mathcal{N} \cup a\mathcal{N}$, 
%where $\mathcal{N}$ is the (possibly infinite) nucleus.
%We then compute the dynamics of $\bar{\phi}$ on $\mathcal{N} \cup a\mathcal{N}$.
The periodic elements of this action are displayed in Table \ref{tbl:phibar} for one representative of each pure Hurwitz class (e.g. each equivalence class of dynamical portrait).

\begin{remark}\label{rmk:OtherHurwitz}
While we have only performed the $\bar{\phi}$ computations for one $Q(4)^*$ map within each pure Hurwitz class,
we can write an element of any other combinatorial class as a twist of the first $Q(4)^*$ map (by Lemma \ref{Hurwitzlem}), and then perform the same method.
For instance, we made no $\bar{\phi}$ computations for 
the rational map $f_-$ defined by the fixed point $\frac{1}{2}(-1-\sqrt{5})$ of the $g$-map $1-z^2$.
However, for the rational map $f_+$ defined by the other fixed point of that $g$-map, we have that 
$f_- \simeq f_+\cdot \alpha$ since $\alpha$ by definition
has non-trivial monodromy action and there are no other $Q(4)^*$ maps in the pure Hurwitz class.
To twist $f_-$ by $h$, we simply use $\bar{\phi}^{\circ n}(\alpha h)$ with $\phibar$ computed from $f_+$
and determine the resulting $\phibar$ cycle as usual.
\end{remark}

\subsection{Identifications}\label{sec:ids}

All that remains is to identify each $\phibar$ cycle with a combinatorial equivalence class of Thurston maps. 
Usually distinct cycles will result in distinct classes, but not always (e.g. see the analysis for $g(z) = z^2$).
The solution to the twisting problem, then, is to apply $\bar{\phi}$ repeatedly to the specified twist
until obtaining one of the periodic elements, each of which is identified with a combinatorial class.

\begin{proof}[Theorem \ref{thm:classification}]
Recall that any map in $Q(2)$ or $Q(3)$ is unobstructed, so from the list of dynamical portraits in Table \ref{tbl:Quad3} one simply computes coefficients to give an enumeration. Each rational map in $Q(4)^*$ can be found the same way (see Table \ref{tbl:Hurwitz1}), and a cross ratio argument shows that any two such maps are not M\"obius conjugate.  From the preceding discussion, all obstructed maps in $Q(4)^*$ must lie in a $\phibar$ cycle, and the remainder of our proof is occupied with analyzing the corresponding Thurston classes (organized by map on moduli space). For all $\phibar$ attractors, the trivial element is identified with the combinatorial class of the $Q(4)^*$ map used in the $\bar{\phi}$ computations.

\begin{itemize}

\item $g(z)=(1-2z)^2$:
the periodic element $\alpha$ is identified with the equivalence class of a Thurston map with obstruction $\beta$.
Twisting about this obstruction produces an infinite family of distinct combinatorial classes of obstructed maps
(see the argument in \cite[Theorem 9.2 V]{KPS}).

\item $g(z)=\frac{1}{1-(1-2z)^2}$: we can distinguish between the combinatorial classes identified with $\alpha$ and with the 2-cycle by considering the finite global attractor (FGA) results proved in the next section.
If $f$ is the $Q(4)^*$ map defined by the fixed point with positive imaginary part (i.e. the $Q(4)^*$ map used to compute the virtual endomorphism $\phi$ from the first coordinate of the wreath recursion), then we can compute the FGA for $f\cdot\alpha$ by changing the virtual endomorphism to instead be associated with the second coordinate.
When we make these calculations, we see that the FGA for $f\cdot\alpha$ has the 2-cycle $\beta\leftrightarrow\gamma$, so $\alpha$ is identified with the combinatorial class of the $Q(4)^*$ map defined by fixed point $-.4196$.
Thus, the combinatorial class of the $Q(4)^*$ map defined by the fixed point with negative imaginary part is identified with the 2-cycle $\gamma^{-1}\leftrightarrow\alpha\gamma$.

\item $g(z)=1-\frac{1}{(1-2z)^2}$: the 2-cycle is identified with the combinatorial class of the $Q(4)^*$ map defined by other fixed point of the $g$-map,
and the $\alpha\beta^n$ are identified with an infinite family of combinatorial classes of obstructed maps as above.

\item $g(z)=z^2$: in the row of Table \ref{tbl:phibar} for the $g$-map $z^2$, 
we have two families of fixed points of $\phibar$: $\beta^n$ and $\alpha^2\beta^n\alpha^{-1}$ for $n\in\Z$. 
In fact, these are both identified with the same family of combinatorial classes of obstructed maps. 
Let $f$ be the (obstructed) $Q(4)^*$ map used to make the $\phibar$ calculations and 
observe that since $\alpha^2 = \left<\alpha, \alpha\right>$ for this wreath recursion, then
$$f\cdot \alpha^2\beta^n\alpha^{-1} = \alpha \cdot f \cdot \beta^n\alpha^{-1}$$
and by conjugation we have that
$$\alpha \cdot f \cdot \beta^n\alpha^{-1} \simeq f\cdot \beta^n$$
Thus, $f\cdot \beta^n$ and $f\cdot \alpha^2\beta^n\alpha^{-1}$ are equivalent maps.

\item  $g(z)=1-z^2$: let $f$ be the $Q(4)^*$ map defined by the fixed point $\frac{1}{2}(-1+\sqrt{5})$ of $g$. We can show that $f \cdot \alpha^2\beta^{-1} \simeq f$
(since $\alpha^2 = \left<\beta, \beta\right>$ in this setting),
and so the 2-cycle is identified with the combinatorial class of $f$.
Thus, $\alpha$ is identified with the combinatorial class of the $Q(4)^*$ map defined by the other fixed point.

\item $g(z)=\frac{1}{z^2}$: the 2-cycle is identified with the combinatorial class of the fixed point of the $g$-map with positive imaginary part,
and the $\alpha\beta^n$ are identified with an infinite family of combinatorial classes of obstructed maps.

\item $g(z)=\frac{1}{1-z^2}$: we again use the virtual endomorphism for the second coordinate to compute the FGA for the map identified with $\alpha$.
This computation reveals it to match the $Q(4)^*$ map defined by the fixed point with negative imaginary part, and so the $Q(4)^*$ map defined by the fixed point with positive imaginary part is identified with the 3-cycle.

\end{itemize}
We have thus shown that each $\phibar$ cycle either corresponds to one of the rational maps already found, or lies in a one-parameter family of obstructed twists found in Table \ref{tbl:Obstructed}.
\end{proof}

\begin{table}
\caption{All obstructed non-Euclidean quadratic Thurston maps with four postcritical points}
\begin{tabular}{c|c|c|l}
\hline
%Headings
$f\in Q(4)^*$ & $\bullet$ & $g$   & dynamical portrait of $f$ \\\hline\hline

%First block

 $T_{\alpha\beta^n}\circ M_{\bullet,0}\circ g$ & $\frac{1}{4}$& & $\xymatrix{
\frac{1}{2}\ar@{=>}[r] &\bullet\ar[r] & 0\ar[r] & \infty \ar@{=>}@/^/[r] & 1 \ar@/^/[l]}$   \\ 
  \cline{1-2}\cline{4-4}%\hline\hline

$T_{\alpha\beta^n}\circ M_{\bullet,1}\circ g$ & $\frac{1}{4}$ & $(1-2z)^2$  & $\xymatrix{\frac{1}{2}\ar@{=>}[r] &\infty\ar@{=>}[r] & 0 \ar[r] &\bullet\ar@/^/[r] & 1 \ar@/^/[l]}$   \\ \cline{1-2}\cline{4-4}

$T_{\alpha\beta^n}\circ M_{\bullet,\infty}\circ g$ & $\frac{1}{4}$&  &  $\xymatrix{\frac{1}{2}\ar@{=>}[r] & 1 \ar[r] & 0 \ar@(ur,dr)[] &\infty\ar@{=>}@/^/[r] & \bullet \ar@/^/[l]}$    \\ \hline

%Second block

$T_{\alpha\beta^n}\circ M_{\bullet,\infty}\circ g$ &   $1 + \frac{1}{2}i$ &  & $\xymatrix{\frac{1}{2}\ar@{=>}[r] &\bullet\ar[r] &\infty\ar@{=>}[r] & 0 \ar[r] & 1 \ar@(ur,dr)[]}$  \\   \cline{1-2}\cline{4-4}

$T_{\alpha\beta^n}\circ M_{\bullet,0}\circ g$ &  $1 + \frac{1}{2}i$ & $1-\dfrac{1}{(1-2z)^2}$  & $\xymatrix{\frac{1}{2}\ar@{=>}[r] & 1\ar[r] & \bullet \ar@/^/[r] & 0 \ar@/^/[l] &\infty\ar@{=>}@(ur,dr)[]}$  \\ \cline{1-2}\cline{4-4}

$T_{\alpha\beta^n}\circ M_{\bullet,1}\circ g$ & $ 1 + \frac{1}{2}i$ &   & $\xymatrix{\frac{1}{2}\ar@{=>}[r] & 0 \ar[r] &\infty \ar@{=>}[r] &\bullet\ar[r] & 1 \ar@/^/[ll]}$    \\  \hline

%Third block
  
$T_{\beta^n}\circ M_{\bullet,\infty}\circ P_{1/4}\circ g$ & See caption & $z^2$  & $\xymatrix{0\ar@{=>}@/^/[r] & 1 \ar@/^/[l]   &\infty\ar@{=>}@/^/[r] &\bullet\ar@/^/[l]}$  \\
\hline

%Fourth block

$T_{\alpha\beta^n}\circ M_{\bullet,0}\circ g$ &  $\frac{1}{2}(-1-\sqrt{3}i)$ & $\frac{1}{z^2}$ & $\xymatrix{
0 \ar@{=>}[r] & 1 \ar[r] &\infty \ar@{=>}[r] &\bullet \ar@/^/[lll]}$  \\
\hline

\end{tabular}
\caption*{Each obstructed combinatorial class $Q(4)^*$ is represented by exactly one element in the first column where $n\in\mathbb{Z}$. The notation $T_\gamma$ denotes the point-push of $\bullet$ along the closed curve $\gamma$ given in terms of the fundamental group (the reader is referred to Section \ref{sec:wreath} for basepoint and generator conventions). We denote by $P_{1/4}$ the homeomorphism defined by a real symmetric point-push carrying $1/16$ to $1/4$ while fixing $\{0,1,\infty\}$ for all time. Thus $\bullet = 1/4$ is fixed by $P_{1/4}\circ g$ where $g(z)=z^2$.}\label{tbl:Obstructed}
\end{table}

\begin{table}%\label{tbl:attractor_FGA}
\begin{tabular}{c|lll}
 % \hline

%Headings
$g$-map & Wreath recursion & Attractor of the $\bar{\phi}$ map &FGA \\\hline\hline

%First block (Hurwitz 1)

% \rowcolor{lightgray}
\hline  

$(1-2z)^2$ & $\alpha = \left<1,1\right>\sigma, \quad \beta=\left<\alpha,\beta\right>$ &  $\xymatrix{1 \ar@(ur,dr)[]}\hspace{.55cm}, \alpha\beta^n\hspace{-.15cm}\xymatrix{ \ar@(ur,dr)[]}\qquad$ for $n\in\Z$ & $\odot$ \\  \cline{2-3}\hline

%Colors individual cells:\cellcolor{gray!50}

%Second block (Hurwitz 1)

&  $\alpha = \left<1,1\right>\sigma, \quad \beta=\left<\delta,\beta\alpha\beta^{-1} \right>$ & ---------  & $\odot,\; \beta\leftrightarrow\gamma$\\ \cline{2-4}
  
$\dfrac{1}{1-(1-2z)^2}$ &  $\alpha = \left<1,1\right>\sigma, \quad \beta=\left<\alpha,\gamma \right>$ & $\xymatrix{1 \ar@(ur,dr)[]}\hspace{.55cm}, \xymatrix{\alpha \ar@(ur,dr)[]}\hspace{.55cm},\hspace{.2cm} \gamma^{-1}\leftrightarrow\alpha\gamma $ & $\odot$ \\\cline{2-4}

&   $\alpha = \left<1,1\right>\sigma, \quad \beta=\left<\alpha,\delta \right>$   & --------- & $\odot$\\ 
  \hline

%Third block (Hurwitz 1)

\multirow{2}{*}{$1-\dfrac{1}{(1-2z)^2}$} &  $\alpha = \left<1,1\right>\sigma, \quad \beta=\left<\gamma,\beta \right>$ & $\xymatrix{1 \ar@(ur,dr)[]}\hspace{.55cm},\hspace{.15cm} \gamma\leftrightarrow\alpha\gamma^{-1},\hspace{.15cm} \alpha\beta^n\hspace{-.15cm}\xymatrix{ \ar@(ur,dr)[]}\hspace{.55cm}$ for $n\in\Z$ & $\odot$ \\   \cline{2-4}

& $\alpha = \left<1,1\right>\sigma, \quad \beta=\left<\delta,\beta \right>$ & --------- & $\odot$ \\ \cline{2-3}%\hline\noalign{\bigskip}\noalign{\medskip}\hline

%First block (Hurwitz 2)
%\hline\hline
\hline
%\rowcolor{lightgray} 
\hline 

$z^2$ & $\alpha = \left<1,\alpha\right>\sigma, \quad \beta=\left<\beta,1\right>$ & $\beta^n\hspace{-.15cm} \xymatrix{ \ar@(ur,dr)[]}\hspace{.6cm},$ $\alpha^2\beta^n\alpha^{-1}\hspace{-.15cm}\xymatrix{ \ar@(ur,dr)[]}\hspace{.6cm}$ for $n\in\Z$  %$\xymatrix{\alpha^2\beta^n\alpha^{-1} \ar@(ur,dr)[]}\quad$ for $n\in\Z$
& N/A
\\ \hline

%\alpha^2\beta^n\alpha^{-1} \ar@(ur,dr)[]}

%Second block (Hurwitz 2)

\multirow{2}{*}{$1-z^2$} &    $\alpha = \left<1,\beta \right>\sigma, \quad \beta=\left<\alpha,1 \right>$ & $\xymatrix{1 \ar@(ur,dr)[]}\hspace{.55cm}, \xymatrix{\alpha \ar@(ur,dr)[]}\hspace{.5cm},\hspace{.1cm} \beta\alpha^{-1}\leftrightarrow\alpha^2\beta^{-1} $ & $\odot,\; \alpha\leftrightarrow\beta, \gamma\leftrightarrow\delta$ \\   \cline{2-4}

& $\alpha = \left<1,\beta\right>\sigma, \quad \beta=\left<1,\beta\alpha\beta^{-1} \right>$ & --------- & $\odot, \xymatrix{\delta \ar@(ur,dr)[]}$\\ \cline{2-3}\hline

%Third block (Hurwitz 2)

\multirow{2}{*}{$\dfrac{1}{z^2}$}&   $\alpha = \left<1,\delta \right>\sigma, \quad \beta=\left<1,\alpha^{-1}\beta\alpha \right>$ & --------- & $\odot,\;\alpha\leftrightarrow\delta$ \\   \cline{2-4}

& $\alpha = \left<1,\gamma\right>\sigma, \quad \beta=\left<1,\beta \right>$ & $\xymatrix{1 \ar@(ur,dr)[]}\hspace{.55cm}, \gamma\leftrightarrow\alpha\gamma^{-1}, \hspace{.1cm}\alpha\beta^n\hspace{-.15cm}\xymatrix{ \ar@(ur,dr)[]}\hspace{.55cm}$ for $n\in\Z$ & $\odot,\;\alpha\leftrightarrow\gamma$ \\ \cline{2-3}\hline

%Fourth block (Hurwitz 2)

&  $\alpha = \left<1,\delta\right>\sigma, \quad \beta=\left<1,\alpha \right>$ & $\xymatrix{1 \ar@(ur,dr)[]}\hspace{.55cm}, \xymatrix{\alpha \ar@(ur,dr)[]}\hspace{.5cm}, \xymatrix{\delta \ar[r] &\gamma^{-1} \ar[r] & \alpha^2 \ar@/^/[ll]}$ & $\odot$ \\ \cline{2-4}
  
$\dfrac{1}{1-z^2}$ &   $\alpha = \left<1,\gamma\right>\sigma, \quad \beta=\left<\alpha,1\right>$ & --------- & $\odot,\xymatrix{\alpha \ar[r] &\gamma \ar[r] & \beta \ar@/^/[ll]}$ \\\cline{2-4}

 &  $\alpha = \left<1,\delta \right>\sigma, \quad \beta=\left<\alpha,1 \right>$   & --------- & $\odot,\xymatrix{\alpha \ar[r] &\delta \ar[r] & \beta \ar@/^/[ll]}$\\ 
  \hline

\end{tabular}
\caption{Wreath recursion for each map on moduli space at each fixed point not in $\{0,1,\infty\}$, the attractor for the non-homomorphism map that is used to solve the twisting problem, and the minimal finite global attractor for the pullback map on curves. Note that we compute the $\phibar$ attractor for only one map in each pure Hurwitz class. This is all that is required for the enumeration; the other attractors can be computed using Remark \ref{rmk:OtherHurwitz} if desired.  Recall that $\gamma = \beta^{-1}\alpha^{-1}$ and $\delta = \alpha^{-1}\beta^{-1}$. We also note that the infinite global attractor for $g(z)=z^2$ is $\{\odot, \alpha^{\beta^n}, \beta, \gamma^{\beta^n}, \delta^{\beta^n}, \beta^{\alpha\beta^n}\ | \ n\in\Z\}$}\label{tbl:phibar}
\end{table}

\section{Pullback on essential curves}\label{sec:PullbackOnCurves}
Recall that the collection of all homotopy classes of essential curves is denoted $\mathcal{C}_f$.

\begin{definition} The \emph{pullback function on curves} $\mu_f:\mathcal{C}_f\cup\{\odot\}\to\mathcal{C}_f\cup\{\odot\}$ is defined as follows:
\begin{itemize}
\item if $\gamma\in\mathcal{C}_f$, and $f^{-1}(\gamma)$ has an essential component $\tilde{\gamma}$, then $\mu_f(\gamma)=\tilde{\gamma}$
\item if $\gamma\in\mathcal{C}_f$, and $f^{-1}(\gamma)$ does not have an essential component, then $\mu_f(\gamma)=\odot$
\item $\mu_f(\odot)=\odot$.
\end{itemize}
\end{definition}

\begin{definition}
A \emph{finite global attractor} of the pullback function on curves $\mu_f$ is a finite, forward invariant set $\mathcal{A}\subset\mathcal{C}_f\cup\{\odot\}$ so that for each $\gamma$, there exists $n$ so that $\mu_f^{\circ n}(\gamma)\in\mathcal{A}$.
\end{definition}

Not every Thurston map has a finite global attractor for its pullback on curves (one example is given below). Pilgrim has conjectured (personal communication) that if $f$ is a rational map with hyperbolic orbifold, then $\mu_f$ has a finite global attractor. The program NETmap \cite{NET} has provided a great deal of evidence in favor of this conjecture, and a number of results have been proven in special cases \cite{TW,Lod,KPS,FKKLPPS}. 

We investigate the global dynamics of the pullback on curves for rational $Q(4)^*$ maps.

\begin{theorem}
\label{thm:FGA}
Let $f$ be a rational $Q(4)^*$ map. The pullback on curves $\mu_f$ has a finite global attractor.
The minimal finite global attractors are exhibited in Table \ref{tbl:phibar}.
\end{theorem}

Our results verify Theorem 8 Part (1) of \cite{FKKLPPS}, where it is shown that the minimal finite global attractor of any rational $Q(4)^*$ map with one critical postcritical point has at most four curves in $\mathcal{C}_f$, using entirely different methods from our own.

\subsection{Computing the pullback on curves}
\label{sec:phihat}

Let $f\in Q(4)^*$ be a rational map. We will compute its pullback on curves in terms of the virtual endomorphism on moduli space using the naturality statement of \cite[Theorem 2.6]{Lod}.  Let $C$ be an essential curve in $(\mathbb{S}^2,P_f)$ and denote by  $G$ the group $\pi_1(\rs\setminus\{0,1,\infty\},z_0)$; recall this is a free group generated by $\alpha$ and $\beta$. The point-pushing isomorphism (see section \ref{sec:wreath}) identifies the right hand Dehn twist about $C$ with a unique parabolic element of $G$
having the form $g=w^{-1}xw$ where $x\in\{\alpha,\beta,\gamma\}$ and $w\in G$.  In case $C$ has an essential preimage, the pullback $\mu_f(C)$ is given by the essential curve corresponding to the parabolic element $\phi((x^n)^w)$, where $n$ is minimal so that the expression is defined (i.e. $w^{-1}x^n w$ is in the domain of $\phi$).  If $C$ has no essential preimage, then $\phi((x^n)^w)$ is trivial.
The computation of the pullback function $\mu_f$ can thus be phrased entirely in terms of group theory.

\begin{comment}
When computing the pullback by $f$ on curves, we make use of the virtual endomorphism $\phi$ that lifts twists by $f$.
Every curve in $(\mathbb{S}^2,P_f)$ is equal to a conjugate of one of the three core curves $\alpha$, $\beta$, or $\gamma = \beta^{-1}\alpha^{-1}$ by some homeomorphism $w$.
We will abuse notation and use $\alpha, \beta, \gamma$ to denote both these curves in $(\mathbb{S}^2,P_f)$ and the (left) Dehn twists about them as homeormorphisms.
We also use $\delta$ to represent $\alpha^{-1}\beta^{-1}$.
By the arguments in \cite{TW}, the pullback of a curve under $f$ can be computed by taking the image under $\phi$ of a power of the associated core curve conjugated by $w$ that lies in the domain of $\phi$.
That is, we compute $\phi((g^n)^w)$ as $g$ ranges over the appropriate powers of core curves and $w$ ranges over the homeomorphisms that leave $P_f$ fixed pointwise.
\end{comment}

We extend the virtual endomorphism $\phi$ to a map (not a homomorphism) $\hat{\phi}:G\rightarrow G$ by
$$\hat{\phi}(w) = \left\{\begin{array}{ll}
\phi(w) & \text{if } w \in \text{Dom}(\phi)\\
\phi(\alpha w) & \text{if } w \in \alpha^{-1}\text{Dom}(\phi).
\end{array}\right.$$
This extension is motivated by the following lemma, which shows how $\phi$ transforms parabolic elements in terms of $\hat{\phi}$.

\begin{lemma}
\label{phihatdef}
Let $x\in\{\alpha,\beta,\gamma\}$ and $w\in G$. For $n$ minimal so that $\phi((x^n)^w)$ is defined,
set $h = x^n$.
Then   
$$\phi(h^{w}) = \left\{\begin{array}{ll}
h_1^{\hat{\phi}(w)} & \text{if } w \in \text{Dom}(\phi)\\
h_2^{\hat{\phi}(w)}& \text{if } w \in \alpha^{-1}\text{Dom}(\phi)
\end{array}\right.$$
where $h_1$ and $h_2$ are the restrictions of $h$. 
%the wreath recursion for $h$ is $\left<h_1, h_2\right>$.

\end{lemma}

\begin{proof}

Let $h=\left<h_1, h_2\right>$ and suppose $w\in \text{Dom}(\phi)$.
Then $$\phi(h^w) = h_1^{\phi(w)}$$
If instead $w\notin \text{Dom}(\phi)$, then $w\in \alpha^{-1}\text{Dom}(\phi)$, 
so there exists $w'\in \text{Dom}(\phi)$ such that $w=\alpha^{-1}w'$.
Then
$$\phi(h^w) = \phi(w'^{-1}\alpha h\alpha^{-1}w') = (\phi(\alpha h\alpha^{-1}))^{\phi(w')}$$
Recall that we chose to make the wreath recursion of $\alpha$ of the form $\alpha=\left<1, \alpha|_2\right>\sigma$.
Thus,
$$(\phi(\alpha h\alpha^{-1}))^{\phi(w')} = (\phi(\left<1, \alpha|_2\right>\sigma \left<h_1,h_2\right>\left<\alpha|_2^{-1}, 1\right>\sigma))^{\phi(w')} = h_2^{\phi(w')}$$

\end{proof}

The next lemma describes the global dynamics of $\hat{\phi}$.  The simplicity of this statement depends heavily on our earlier choice that $\alpha|_1=1$ for every single map.

\begin{lemma}
\label{phihatNlem}

Let $w\in G$.
There exists $N$ such that for $n>N$, we have that
$$\hat{\phi}^{\circ n}(w) \in \mathcal{N}$$
where $\mathcal{N}$ is the (possibly infinite) nucleus of the self-similar action associated with $\phi$.

\end{lemma}

\begin{proof}

If $w=\left<w_1,w_2\right>\in \text{Dom}(\phi)$, then $\hat{\phi}(w) = w_1$.
If $w=\left<w_1,w_2\right>\sigma\in \alpha^{-1}\text{Dom}(\phi)$, then
$$\hat{\phi}(w) = \phi(\alpha w) = \phi(\left<1,\alpha|_2\right>\sigma \left<w_1, w_2\right>\sigma) = w_2$$
Thus, $\hat{\phi}(w)$ is always equal to a restriction of $w$.
\end{proof}

\begin{proof}[Theorem \ref{thm:FGA}]

To determine the global dynamics of the pullback on curves, by Lemmas \ref{phihatdef} and \ref{phihatNlem} we need only examine the dynamics of $\phi$ on the set $\{(x^n)^w|x\in\{\alpha,\beta,\gamma\},n\in\mathbb{Z}, w\in\mathcal{N}^{per}\}$ where $\mathcal{N}^{per}$ is the set of periodic elements of $\hat{\phi}$ in $\mathcal{N}$. The results of these computations are given in Table \ref{tbl:phibar}.

%\begin{lemma}
%\label{conjlem}
%
%Rational maps defined by complex conjugate points in moduli space fixed by the $g$ map will have same behavior for the pullback on curves.
%
%\end{lemma}
%
%\gak{Proof?}

For example, take the virtual endomorphism $\phi$ arising from the wreath recursion in the last row of Table \ref{tbl:phibar}:
$$\phi(\alpha^2) = \delta\qquad \phi(\beta)=\alpha\qquad\phi(\beta^\alpha) = 1$$
Checking the dynamics of $\hat{\phi}$ on $\mathcal{N}$, we see that all elements eventually map to the identity.
So, we need only consider the powers of $\alpha$, $\beta$, and $\delta$ 
(we use $\delta$ here in place of $\gamma$ since it arises naturally in the computations with $\phi$).
Further, since the iteration under $\hat{\phi}$ always lands in the domain of $\phi$, we need only consider the restriction under the first coordinate (i.e. the image under $\phi$).
Even powers of $\alpha$ are taken to powers of $\delta$,
powers of $\beta$ are taken to powers of $\alpha$, and
even powers of $\delta$ are taken to powers of $\beta$.
Thus, we have the finite global attractor as in Table \ref{tbl:phibar}.
\end{proof}

\end{document}